\documentclass[10pt]{amsart}

\usepackage{psfrag}
\usepackage{amsmath}
\usepackage{amssymb}
\usepackage[english]{babel}
\usepackage{amsthm}
\usepackage{amsfonts}
\usepackage[applemac]{inputenc}
\usepackage{tikz}
\usetikzlibrary{arrows}
\usetikzlibrary{patterns}
\usetikzlibrary{shapes}
\usepackage{mathrsfs}
\usepackage{bbm}

\newtheorem{teo}{Theorem}
\newtheorem{lemma}[teo]{Lemma}

\newtheorem{coro}[teo]{Corollary}

\newcommand{\R}{\mathbb{R}}
\newcommand{\ii}{\mathrm{i}}

\newcommand{\eps}{\epsilon}
\newcommand{\N}{\mathbb{N}}
\newcommand{\C}{\mathbb{C}}
\newcommand{\CP}{\mathbb{C}\textrm{P}}
\newcommand{\RP}{\mathbb{R}\mathrm{P}}

\newcommand{\Z}{\mathbb{Z}}

\theoremstyle{remark} 
\newtheorem{remark}[]{Remark}

\newtheorem{example}[]{Example}
\title[A real viewpoint on the intersection of complex quadrics]{A real viewpoint on the intersection of complex quadrics and its topology}

\author{ A. Lerario}
\thanks{SISSA, Trieste}

\begin{document}

\maketitle

\begin{abstract} 
We study the relation between a complex projective set $C\subset \CP^{n}$ and the set $R\subset \RP^{2n+1}$ defined by viewing each equation of $C$ as a pair of real equations. Once $C$ is presented by \emph{quadratic} equations, we can apply a spectral sequence to efficiently compute the homology of $R;$ using the fact that the $\Z_{2}$-cohomology of $R$ is a free $H^{*}(C)$-module with two generators we can in principle reconstruct the homology of $C.$ Explicit computations for the intersection of two complex quadrics are presented.\end{abstract}

\section{Introduction}
Given a projective algebraic set $C\subset\CP^{n}$ we are interested in the computation of its $\Z_{2}$-Betti numbers. The approach we propose is that of studying first the topology of the set  $R\subset \RP^{2n+1},$ defined by viewing each equation of $C$ as a pair of real equations, and then recover the Betti numbers of $C$ from those of $R$ using the formula $b_{j}(C)=\sum_{k=0}^{j}(-1)^{k}b_{j-k}(R).$ In the case $C$ is cut by quadrics, the same is true for $R$ and the homology of the former can be computed using the spectral sequence discussed in \cite{AgLe}. If $C$ is the intersection of two quadrics the computations are quite easy and we devote the last section to perform some of them. \\
In the case $C$ is a \emph{complete} intersection of two quadrics, its geometric properties are studied in \cite{Reid}. The classification of pencil of complex quadrics is a well known fact and involves rather complicated algebraic stuff; the reader can see \cite{Ho} for a classical algebraic treatment or \cite{Di} for a more geometric approach. The reduction to a canonical form for a pair of quadrics is studied in \cite{Th}.\\
Our interest is only in the \emph{rough} topology of $C$ and thus our approach is far from the previous analytic ones, but since we do not make regularity assumptions it includes the treatment of very degenerate objects.

\section{Remarks on real and complex projective sets}

We start by considering the bundle
$$ S^{1}\to \RP^{2n+1}\stackrel{\pi}{\longrightarrow} \CP^{n}$$
where the map $\pi$ is given by $[x_{0},y_{0},\ldots,x_{n},y_{n}]\mapsto [x_{0}+iy_{0},\ldots, x_{n}+iy_{n}].$ 
The fiber of $\pi$ over a point $[v]\in \CP^{n}$ equals the projectivization of the two dimensional real vector space $\textrm{span}_{\R}\{v,iv\}\subset \C^{n+1}\simeq \R^{2n+2}.$ Thus $\RP^{2n+1}$ is the total space of the projectivization of the tautological bundle $O(-1)\to\CP^{n}$ view as a rank two real vector bundle. Applying Leray-Hirsch we get a cohomology class $x\in H^{1}(\RP^{2n+1};\Z_{2}),$ which restricts to a generator of the cohomology of each fiber, such that the map $\alpha \otimes p(x)\mapsto \pi^{*} \alpha \smile p(x),$ where $p\in \Z_{2}[x]/(x^{2})$ and $\alpha\in H^{*}(\CP^{n};\Z_{2})= \Z_{2}[\alpha]/(\alpha^{n+1}),$ gives an isomorphism of $H^{*}(\CP^{n};\Z_{2})$-modules
$$H^{*}(\CP^{n};\Z_{2})\otimes \{1,x\}\simeq H^{*}(\RP^{2n+1};\Z_{2}).$$ 
In particular this tells that $\pi^{*}$ is injective with image the even dimensional cohomology (recall that $|\alpha|=2$).\\
The following geometric description of the map $\pi$ also gives an alternative proof of the previous statement.
Consider the restriction of $\pi$ to $\{[x_{0},y_{0},x_{1},0,\ldots,0]\}\simeq \RP^{2}:$ we see that it maps $\RP^{2}$ to $\{[z_{0},z_{1},0,\ldots,0]\}\simeq \CP^{1}$ trough a homeomoprhism $\{x_{1}\neq0\}\simeq \{z_{1}\neq 0\}$ and by collapsing the line at infinity $\{x_{1}=0\}$ to the point $[1,0,\ldots,0].$ It follows that the modulo 2 degree of $\pi|_{\RP^{2}}$ is one. Using the isomorphism $H^{*}(\RP^{2n+1};\Z_{2})\simeq \Z_{2}[\beta]/(\beta^{2n+2}), \, |\beta|=1,$ we see that $$\pi^{*}:H^{*}(\CP^{n};\Z_{2})\to H^{*}(\RP^{2n+1};\Z_{2})$$ is given by $\alpha\mapsto \beta^{2},$ where $\beta|_{\RP^{2}}$ generates $H^{1}(\RP^{2};\Z_{2}).$
If we consider the Gysin sequence with $\Z_{2}$ coefficients for $\pi,$ then the injectivity of $\pi^{*}$ implies that for every $j$ the following portion of the sequence is exact $$0\to H^{j}(\CP^{n};\Z_{2})\stackrel{\pi^{*}}{\longrightarrow}H^{j}(\RP^{2n+1};\Z_{2})\to H^{j-1}(\CP^{n};\Z_{2})\stackrel{\smile  e}{\longrightarrow} 0$$ where $e=e(\pi)$ is the modulo 2 euler class of $\pi$, which of course turns out to be zero.\\
Let now $I\subset \C[z_{0},\ldots, z_{n}]$ be a homogeneous ideal; we will denote by $C=C(I)$ its zero locus in $\CP^{n}.$ If we restrict the bundle $O(-1)\to \CP^{n}$ to $C$ we get a bundle:
$$\begin{tikzpicture}[xscale=2, yscale=1.5]

    \node (A0_0) at (0, 0) {$C$};
    \node (A1_0) at (1, 0) {$\CP^{n}$};
    \node (A0_1) at (0, 1) {$E$};
    \node (A1_1) at (1, 1) {$O(-1)$};
    \path (A0_0) edge [->] node [auto] {$$} (A1_0);
    \path (A0_1) edge [->] node [auto,swap ] {$$} (A0_0);
    \path (A0_1) edge [->] node [auto] {} (A1_1);
    \path (A1_1) edge [->] node [auto] {$$} (A1_0);
      \end{tikzpicture}
$$
and if we consider the previous as rank two real vector bundles and take their projectivization we get:
$$\begin{tikzpicture}[xscale=2, yscale=1.5]

    \node (A0_0) at (0, 0) {$C$};
    \node (A1_0) at (1, 0) {$\CP^{n}$};
    \node (A0_1) at (0, 1) {$R$};
    \node (A1_1) at (1, 1) {$\RP^{2n+1}$};
    \node (F) at (-0.7,1) {$\RP^{1}$};
    \path (A0_0) edge [->] node [auto] {$i_{C}$} (A1_0);
    \path (A0_1) edge [->] node [auto,swap ] {$\pi|_{R}$} (A0_0);
    \path (A0_1) edge [->] node [auto] {$i_{R}$} (A1_1);
    \path (A1_1) edge [->] node [auto] {$\pi$} (A1_0);
    \path (F) edge [->] node [auto] {} (A0_1);
      \end{tikzpicture}
$$
where $i_{R}$ and $i_{C}$ are the inclusion maps.\\
It is clear that $R$ is an algebraic subset of $\RP^{2n+1}$ whose equations are given by considering each polynomial $f\in I$ as a pair of polynomials $f^{a}=\textrm{Re}(f),f^{b}=\textrm{Im}(f)\in \R[x_{0},y_{0},\ldots,x_{n},y_{n}].$ Applying Leray-Hirsch to $\pi|_{R},$  or the the identity $e(\pi|_{R})=e(\pi)|_{C}=0,$ we get the isomoprhism of $H^{*}(C;\Z_{2})$-modules:
$$H^{*}(C;\Z_{2})\otimes \{1, x|_{R}\}\simeq H^{*}(R, \Z_{2}).$$
The previous isomorphism allows us to compute $\Z_{2}$-Betti numbers of $C$ once those of $R$ are known, via the following formula:
$$b_{j}(C;\Z_{2})=\sum_{k=0}^{j}(-1)^{k}b_{j-k}(R;\Z_{2}).$$
We have the following equalities for the Stiefel-Whitney classes of $E,$ which come from the fact that $E$ is the realification of a complex bundle: $w_{2k}(E)=c_{k}(E)\,\textrm{mod}\,2,$ where $c_{k}$ is the $k$-th Chern class of $E$ seen as a complex bundle, and $w_{2k+1}(E)=0.$\\
Since $E$ has real rank two we have:
$$w_{2}(E)=i_{C}^{*}z \quad\textrm{and}\quad w_{i}(E)=0, \, i\neq 0,2,$$
where $z$ is the generator of $H^{2}(\CP^{n},\Z_{2})$ and we have used the equalities $w_{2}(E)=c_{1}(E)=i_{C}^{*}c_{1}(O(-1))=i_{C}^{*}z.$ \\
The following lemma relates the homomorphisms $i_{C}^{*}$ and $i_{R}^{*}$.
\begin{lemma}\label{coogen}
There exists an odd $r$ such that $(i_{R}^{*})_{k}:H^{k}(\RP^{2n+1};\Z_{2})\to H^{k}(R;\Z_{2})$ is injective for $k\leq r$ and zero for $k>r.$ Moreover for every $k$ we have
$$\textrm{rk}(i_{C}^{*})_{2k}=\textrm{rk}(i_{R}^{*})_{2k}=\textrm{rk}(i_{R}^{*})_{2k+1}.$$
\end{lemma}
\begin{proof}
Let $a$ be such that $(i_{R}^{*})_{a}\equiv 0;$ then using the cup product structure of $H^{*}(\RP^{2n+1};\Z_{2})=\Z_{2}[\beta]/(\beta^{2n+2}),$ we have
$$i_{R}^{*}\beta^{a+k}=i_{R}^{*}\beta^{a}\smile i_{R}^{*}\beta^{k}=0.$$
For the second part of the statement notice that $R=P(E)\stackrel{i_{R}}{\longrightarrow}\RP^{2n+1}$ is linear on the fibres and thus, letting $y=i_{R}^{*}\beta$ we have $y^{2}=(w_{2}(E)+w_{1}(E)y)=w_{2}(E)y$ (since $w_{1}(E)$ is zero), where we interpret  $w_{i}(E)$ as a class on $R$ via $\pi|_{R}^{*}.$ It follows that
$$y^{2k}=w_{2}(E)^{k}\quad \textrm{and} \quad y^{2k+1}=w_{2}(E)^{k}y.$$
On the other hand, since $w_{2}(E)=i_{C}^{*}z,$ then the conclusion follows.
\end{proof} 
\section{The quadratic case}
In this section we study the topology of $R$ in the case $C$ is cut by quadrics, i.e.
$$C=V_{\CP^{n}}(q_{0},\ldots, q_{l}),\quad q_{0},\ldots, q_{l}\in \C[z_{0},\ldots, z_{n}]_{(2)}$$
For a given $q\in\C[z_{0},\ldots, z_{n}]_{(2)},\, q(z)=z^{T}Qz$ with $Q=A-iB$ and  $A,B\in \textrm{Sym}(n+1,\R)$ we define the symmetric matrix 
$$P=\left( \begin{smallmatrix}A&B\\B&-A\end{smallmatrix}\right).$$
We set $J=\left( \begin{smallmatrix}0&I\\-I&0\end{smallmatrix}\right)$  (it is a $(n+2)\times(n+2)$ matrix) and given $q_{0},\ldots, q_{l}\in \C[z_{0},\ldots,z_{n}]_{(2)}$ we define $p:S^{2l+1}\to\textrm{Sym}(2n+2,\R)$ by
$$(a_{0},b_{0},\ldots, a_{l},b_{l})\stackrel{p}{\mapsto}a_{0}P_{0}-b_{0}JP_{0}+\cdots+a_{l}P_{l}-b_{l}JP_{l}.$$
For every polynomial $f\in \C[z_{0},\ldots, z_{n}]$ recall that we have defined the polynomials $f^{a},f^{b}\in \R[x_{0},y_{0},\ldots, x_{n},y_{n}]$ by
$$f^{a}(x,y)=\textrm{Re}(f)(x+iy),\quad \quad f^{b}(x,y)=\textrm{Im}(f)(x+iy).$$
Thus if $C=V_{\CP^{n}}(q_{0},\ldots, q_{l})$, we have
$$R=V_{\RP^{2n+1}}(q_{0}^{a},q_{0}^{b},\ldots, q_{l}^{a},q_{l}^{b})$$
We easily see that $\ii^{+}(a_{0}q_{0}^{a}+b_{0}q_{0}^{b}+\cdots+a_{l}q_{l}^{a}+b_{l}q_{l}^{b})=\ii^{+}(p(a_{0},b_{0},\ldots, b_{l},q_{l})):$ this is simply because $P_{j}$ and $-JP_{j}$ are the symmetric matrices associated respectively to the quadratic forms $q_{j}^{a}$ and $q_{j}^{b}.$\\
Following \cite{AgLe} for every $j\in \N$ we define 
$$\Omega^{j}=\{\alpha \in S^{2l+1}\, |\, \ii^{+}(p(\alpha))\geq j\}$$
and if we let $B$ be the unit ball in $\R^{2l+2},\, \partial B=S^{2l+1}$ we recall the existence of a first quadrant spectral sequence $(E_{r}, d_{r})_{r\geq 0}$ such that: $$(E_{r}, d_{r})\Rightarrow H_{2n+1-*}(R;\Z_{2}),\quad E_{2}^{i,j}=H^{i}(B, \Omega^{j+1};\Z_{2}).$$
For $j\in \N$ if we let $P^{j}\subset S^{2l+1}\subset \C^{l+1}$ be defined by
$$P^{j}=\{(\alpha_{0},\ldots, \alpha_{l})\in S^{2l+1}\, |\, \textrm{rk}_{\C}(\alpha_{0}q_{0}+\ldots+\alpha_{l}q_{l})\geq j\}$$
we can rewrite Theorem A of \cite{AgLe} in the following more natural way.
\begin{teo}\label{teospec}There exists a cohomology spectral sequence of the first quadrant $(E_{r},d_{r}),$ converging to $H_{2n+1-*}(R, \Z_{2})$ such that
$E_{2}^{i,j}=H^{i}(B, P^{j+1};\Z_{2}).$
\end{teo}
\begin{proof}
We will prove that for every $j$ the two sets $P^{j+1}$ and $\Omega^{j+1}$ are homeomorphic, and in fact if $\tau:\C^{l+1}\to \C^{l+1}$ denotes complex coniugation that we have
$$\tau(P^{j+1})=\Omega^{j+1}.$$
If we use the matrix notation for each $q_{j}$ we have $q_{j}(z)=z^{T} Q_{j}z $ for $Q_{j}\in \textrm{Sym}(n+1, \C)$ and writing $Q_{j}=A_{j}-iB_{j}$ with $A_{j},B_{j}\in \textrm{Sym}(n+1, \R)$
$$q_{j}^{a}(x,y)= \langle \left( \begin{smallmatrix}x\\y\end{smallmatrix}\right) , \left( \begin{smallmatrix} A_{j}&B_{j}\\B_{j}&-A_{j}\end{smallmatrix}\right)\left( \begin{smallmatrix}x\\y\end{smallmatrix}\right)\rangle,\quad q_{j}^{b}(x,y)= \langle \left( \begin{smallmatrix}x\\y\end{smallmatrix}\right) , \left( \begin{smallmatrix} -B_{j}&A_{j}\\A_{j}&B_{j}\end{smallmatrix}\right)\left( \begin{smallmatrix}x\\y\end{smallmatrix}\right)\rangle .$$
In particular notice that the matrix associated to the real quadratic form $a_{0}q_{0}^{a}+b_{0}q_{0}^{b}+\cdots +a_{l}q_{l}^{a}+b_{l}q_{l}^{b}$ is of the form $$M=\left( \begin{smallmatrix} A&B\\B&-A\end{smallmatrix}\right)$$ for $A,B \in \textrm{Sym}(n+1,\R).$ If $\lambda$ is an eigenvalue of $M$ and $V_{\lambda}$ is the corresponding eigenspace, then the map $(u,v)\mapsto (-v,u)$ gives an isomorphism $V_{\lambda}\simeq V_{-\lambda}$. This implies 
$$2\ii^{+}(M)=\textrm{rk}_{\R}(M).$$
On the other side it is easy to show that
$$\textrm{rk}_{\R}\left( \begin{smallmatrix} A&B\\B&-A\end{smallmatrix}\right)=2\textrm{rk}_{\C}(A-iB)$$
in fact the map $(u,v)\mapsto u+iv$ gives an isomorphism of real vector space $\ker(M)\simeq \ker(A-iB).$
Comparing now  the matrices associated to $a_{0}q_{0}^{a}+b_{0}q_{0}^{b}+\cdots +a_{l}q_{l}^{a}+b_{l}q_{l}^{b}$ and to $(a_{0}-ib_{0})q_{0}+\ldots+(a_{l}-ib_{l})q_{l}$ we get the result.

\end{proof}

\begin{remark}Even more natural than the sets $\{P^{j}\}_{j\in \N}$ are the sets
$$Y^{j}=\{[\alpha]\in \CP^{l},\, \alpha \in S^{2l+1} |\, \textrm{rk}(p(\alpha))\geq j\}.$$
If we consider the hopf bundle $S^{1}\to S^{2l+1}\stackrel{h}{\longrightarrow}\CP^{l}$
we see that $h(P^{j})=Y^{j}$ and thus $$P^{j}\neq S^{2l+1}\Rightarrow H^{*}(P^{j})=H^{*}(Y^{j})\otimes H^{*}(S^{1}).$$
In this way we see that it is possible to express all the data  for $E_{2}$ of the previous spectral sequence only in terms of the linear system $\mathbb{P}(\textrm{span}(q_{0},\ldots, q_{l}))\subset \mathbb{P}(\C[z_{0},\ldots, z_{n}]_{(2)})$.
\end{remark}
We recall now from \cite{AgLe} the following description for  the second differential of $(E_{r},d_{r})_{r\geq 0}$\\
For each $P\in \textrm{Sym}(2n+2,\R)$ we order the eigenvalues of $P$ in increasing way:
$$\lambda_{1}(P)\geq \cdots \geq \lambda_{2n+2}(P)$$
and we define $$D_{j}=\{\alpha \in S^{2l+1}\, |\, \lambda_{j}(p(\alpha))\neq \lambda_{j+1}(p(\alpha))\}.$$
Then there is a naturally defined bundle $\R^{j}\to L_{j}\to D_{j}$ whose fiber over  a point $\alpha \in D_{j}$ equals $(L_{j})_{\alpha}=\textrm{span}\{v\in \R^{2n+2}\, |\, p(\alpha)v=\lambda_{i} v,\, i=1, \ldots, j\}$ and whose vector bundle structure is given by the inclusion $L_{j}\hookrightarrow D_{j}\times \R^{2n+2}.$ We define $w_{1,j}\in H^{1}(D_{j})$ to be the first Stiefel-Whitney class of $L_{j}$ and 
$$\gamma_{1,j}=\partial^{*}w_{1,j}\in H^{2}(B, D_{j})$$ where $\partial^{*}:H^{1}(D_{j})\to H^{2}(B, D_{j})$ is the connecting isomorphism. In \cite{AgLe} it is proved that $d_{2}:H^{i}(B, \Omega^{j+1})\to H^{i+2}(B, \Omega^{j})$ is given by
$$d_{2}(x)=(x\smile \gamma_{1,j})|_{(B, \Omega^{j})}$$
(notice that $\Omega^{j}\subset \Omega^{j+1}\cup D_{j}$).\\
If we let $[\alpha_{0},\ldots,\alpha_{l}]=[a_{0}+ib_{0},\ldots, a_{l}+ib_{l}]\in \CP^{l}$ such that $\alpha=(a_{0},b_{0},\ldots, a_{l},b_{l})\in S^{2l+1}$ then $p|_{h^{-1}[\alpha_{0},\ldots, \alpha_{l}]}:S^{1}\to \textrm{Sym}(2n+2,\R)$ equals
$$\theta\mapsto  \left( \begin{smallmatrix} I \cos \theta&-I \sin \theta \\ I \sin \theta& I \cos \theta\end{smallmatrix}\right)p(\alpha)$$
as one can easily check. The following lemma is the main ingredient for the explicit computations of $d_{2}.$

\begin{lemma}\label{split}Let $A,B, I\in \emph{\textrm{Sym}}(n+1,\R),$ with $I$ the identity matrix, $R({\theta})=\left( \begin{smallmatrix} I \cos \theta&-I \sin \theta \\ I \sin \theta& I \cos \theta\end{smallmatrix}\right)$ and $M=\left( \begin{smallmatrix} A&B \\ B& -A\end{smallmatrix}\right).$ Let $c:S^{1}\to \emph{\textrm{Sym}}(2n+2,\R)$ be defined by
$$\theta \mapsto R({\theta })M.$$
Consider the bundle $c^{*}L$ over $S^{1}$ whose fibre at the point $\theta\in S^{1}$ is $$(c^{*}L)_{\theta}=\emph{\textrm{span}}\{w\in \R^{2n+2}\,|\, \exists \lambda >0 \, :\, c(\theta)w=\lambda w\} $$ and whose vector bundle structure is given by its inclusion in $S^{1}\times \R^{2n+2}.$ Then the following holds for the first Stiefel-Whitney class of $c^{*}L:$
$$w_{1}(c^{*}L)=\emph{\textrm{rk}}_{\C}(A-iB)\,\emph{\textrm{mod}}\,2.$$
\end{lemma}

\begin{proof}
First notice that if $w=\left( \begin{smallmatrix} u \\v \end{smallmatrix}\right)$ is an eigenvector of $\left( \begin{smallmatrix} A&B \\ B& -A\end{smallmatrix}\right)$ for the eigenvalue $\lambda,$ then $Jw=\left( \begin{smallmatrix} v \\-u \end{smallmatrix}\right)$ is an eigenvector for the eigenvalue $-\lambda.$ It follows that there exists a basis $\{w_{1},Jw_{1},\ldots, w_{n+1},Jw_{n+1}\}$ of $\R^{2n+2}$ of eigenvectors of $\left( \begin{smallmatrix} A&B \\ B& -A\end{smallmatrix}\right)$ such that $\left( \begin{smallmatrix} A&B \\ B& -A\end{smallmatrix}\right)w_{j}=\lambda_{j}w_{j}$ with $\lambda_{j}\geq 0.$ Let now $W_{j}=\textrm{span}\{w_{j},Jw_{j}\}.$ Then $W_{j}$ is $R(\theta)$-invariant:
$R(\theta)w_{j}=\cos \theta w_{j}-\sin \theta Jw_{j}$ and $R(\theta)Jw_{j}=\sin \theta w_{j}+\cos \theta Jw_{j}.$
Thus, using the above basis, we see that $R(\theta)$ is congruent to
$$M^{T}R(\theta)M=\textrm{diag}(D_{1}(\theta),\ldots, D_{n+1}(\theta)),\,\,\,D_{j}(\theta)=\lambda_{j }\left( \begin{smallmatrix} \cos \theta&\sin \theta \\ \sin \theta& -\cos \theta\end{smallmatrix}\right)$$
If $c_{j}:\theta \mapsto D_{j}(\theta),$ then clearly we have the splitting $c^{*}L=c_{1}^{*}L\oplus\cdots \oplus c_{n+1}^{*}L.$
Since $w_{j}(c^{*}L)=0$ if and only if $\lambda_{j}=0,$ then
$$w_{1}(c^{*}L)=\frac{1}{2}\textrm{rk}_{\R}\left( \begin{smallmatrix} A&B \\ B& -A\end{smallmatrix}\right)=\textrm{rk}_{\C}(A-iB)$$
where the last equality comes from the proof of Theorem \ref{teospec}.
\end{proof}

\begin{coro}[The cohomology of one single quadric]
Let $q\in \C[z_{0},\ldots,z_{n}]_{(2)}$ be a quadratic form with $\emph{\textrm{rk}}(q)=\rho >0$ and $$C=V(q)\subset \emph{$\CP^{n}$}$$ Then the Betti numbers of $C$ are:
$$\rho \,\textrm{even:}\quad b_{j}(C)=\left\{ \begin{array}{ll} 0&\textrm{if $j$ is odd;}\\
1 &\textrm{if $j$ is even, $0\leq j\leq 2n-2,\, j\neq 2n-\rho$}\\
2& \textrm{if $j=2n-\rho$}\end{array}\right.$$
$$
\rho \,\textrm{odd:}\quad b_{j}(C)=\left\{ \begin{array}{ll} 0&\textrm{if $j$ is odd;}\\
1 &\textrm{if $j$ is even, $0\leq j\leq 2n-2$}
\end{array}\right.
$$

\end{coro} 
\begin{proof}
We first compute $H^{*}(R)$ using Theorem \ref{teospec}: in this case if $Q=A-iB,$ then $R$ is the intersection of the two quadrics defined by the symmetric matrices $P=\left( \begin{smallmatrix} A&B \\ B& -A\end{smallmatrix}\right)$ and $-JP=\left( \begin{smallmatrix} B&A \\ A& -B\end{smallmatrix}\right)$ and $p:S^{1}\to \textrm{Sym}(2n+2,\R)$ equals $\theta \mapsto R_{\theta} \left( \begin{smallmatrix} A&B \\ B& -A\end{smallmatrix}\right).$ The function $\ii^{+}$ has constant value $\rho$ and thus the $E_{2}$ table for $R$ has the following picture:
$$
\begin{array}{c|c|c|c|}
2n+1&\Z_{2}&0&0\\
\hline
&\vdots&\vdots&\vdots\\
\rho&\Z_{2}&0&0\\
\hline
&0&0&\Z_{2}\\
&\vdots&\vdots&\vdots\\
&0&0&\Z_2\\
\hline
\end{array}
$$
The only (possibly) nonzero differential is 
$$d_{2}:E_{2}^{0,\rho}\to E_{2}^{2, \rho-1}$$
which by the previous discussion equals $1\mapsto\partial ^{*}w_{1}(p^{*}L).$ Lemma \ref{split} implies now
$$d_{2}(1)=\rho\, \textrm{mod}\, 2 $$
Thus if $\rho$ is even $E_{2}=E_{\infty}$ and if $\rho$ is odd the $(\rho-1)$-th and the $\rho$-th row of $E_{3}=E_{\infty}$ are zero. Applying the formula $b_{j}(C;\Z_{2})=\sum_{k=0}^{j}(-1)^{k}b_{j-k}(R;\Z_{2})$ gives the result.

\end{proof}
Using the previous spectral sequence we can easily compute the rank of the map induced on the $\Z_{2}$-cohomology by the inclusion
$$i_{C}:C\hookrightarrow \CP^{n}.$$ We recall that from Theorem C of \cite{AgLe} we have $\textrm{dim}(E_{\infty}^{0,2n+1-k})=\textrm{rk}(i_R^{*})_{k};$ thus applying Lemma \ref{coogen} we get the following.

\begin{teo}\label{coo}For every $k$ we have
$$\emph{\textrm{rk}}(i_{C}^{*})_{2k}=\emph{\textrm{dim}}E_{\infty}^{0, 2n+1-2k}$$
and the zeroth column of $E_{\infty}$ must be the following:
$$
E_{\infty}^{0,*}=\begin{array}{|c|}
\hline
\Z_{2}\\

\vdots\\

\Z_{2}\\
\hline
0\\

\vdots \\

0\\
\hline
\end{array}
$$
where the number of $\Z_{2}$ summand is an even number $r+1,$ and $r$ is that given by Lemma \ref{coogen}.
\end{teo}

Notice in particular that $E_{\infty}^{0,2a}=\Z_{2}$ iff $E_{\infty}^{0,2a+1}=\Z_{2}.$

\section{The intersection of two complex quadrics}

We apply here the previous result to compute the cohomology of the intersection of two complex quadrics:
$$C=V(q_{0},q_{1})\subset \CP^{n}.$$
We define $\Sigma_{j}=\{[\alpha]\in \CP^{1}\, |\, \textrm{rk}(\alpha_{0}q_{0}+\alpha_{1}q_{1})\leq j-1\};$ for $j\leq \mu=\max \ii^{+}$
 we see that $\Sigma_{j}$ consists of a finite number of points (it is a proper algebraic subset) $\Sigma_{j}=\{[\alpha_{1}],\ldots [\alpha_{\sigma_{j}}]\},$ where we have set 
 $$\sigma_{j}=\textrm{card} (\Sigma_{j}), \quad j\leq \mu.$$
 The discussion of the previous sections implies that for every $[\alpha]\in \CP^{1}$ the function $\ii^{+}$ is constant on the circle $h^{-1}[\alpha]\subset S^{3}$ with value
 $$\ii^{+}|_{h^{-1}[\alpha]}\equiv \textrm{rk}(\alpha_{0}q_{0}+\alpha_{1}q_{1})=\rho([\alpha]).$$
 Thus it is defined the bundle $\R^{\rho([\alpha])}\to L_{[\alpha]}\to h^{-1}[\alpha]$ of positive eigenspace of $p|_{h^{-1}[\alpha]}$ and Lemma \ref{split} implies
 $$w_{1}(L_{[\alpha]})=\rho([\alpha])\, \textrm{mod}\, 2.$$
 Fore every $[\alpha]\in \CP^{1}$ we let $m_{[\alpha]}$ be the multiplicity of $[\alpha]$ as a solution of $\det(\alpha_{0}Q_{0}+\alpha_{1}Q_{1})=0;$
 notice that in general $n+1-\rho ([\alpha])\neq m_{[\alpha]}.$\\ For every $j\in \N$ we see that
 $$\Omega^{j+1}=S^{3}\backslash h^{-1}(\Sigma_{j+1})$$
 If we let $\nu$ be the minimum of $\ii^{+}$ over $S^{3},$ we see that for $i>0$ and $\nu+1\leq j+1\leq \mu$
 $$E_{2}^{i,j}=H^{i}(B, S^{3}\backslash h^{-1}(\Sigma_{j+1}))\simeq \tilde{H}_{3-i}( h^{-1}(\Sigma_{j+1}))=\left\{ \begin{array}{ll} 0&\textrm{if $i\neq2,3$;}\\
\Z_{2}^{\sigma_{j+1}} &\textrm{if $i=2$}\\
\Z_{2}^{\sigma_{j+1}-1}& \textrm{if $i=3$}\end{array}\right.$$
This gives the following picture for the table of ranks of $E_{2}:$

$$
\begin{array}{c|c|c|c|c|c|}
2n+1&1&0&&&\\
&\vdots&\vdots&&&\\
&1&0&&&\\
\mu&1&0&&&\\
\hline
\mu-1&0&0&\sigma_{\mu}&\sigma_{\mu}-1&0\\
&\vdots&\vdots&\vdots&\vdots&\vdots\\
\nu&0&0&\sigma_{\nu+1}&\sigma_{\nu+1}-1&0\\
\hline
&0&0&0&0&1\\
&\vdots&\vdots&\vdots&\vdots&\vdots\\
&0&0&0&0&1\\
\hline

\end{array}
$$

We proceed now with the computation of the second differential; the only two possibly nonzero differential are $d_{2}^{0,\mu}$ and $d_{2}^{2, \nu},$ for which the following theorem holds; for an integer $m$ we let $\overline{m}\in \Z_{2}$ be its residue modulo $2.$

\begin{teo}\label{d2}The following formula holds for the differential $d_{2}^{2, \nu}:\Z_{2}^{\sigma_{\nu+1}}\to \Z_{2}$
$$d_{2}^{2, \nu}(x_{1}, \ldots, x_{\sigma_{\nu+1}})=\overline{\nu} \sum_{k=1}^{\sigma_{\nu+1}} x_{k}.$$
Moreover in the case $\mu=n+1,$ we also have the following explicit expression for $d_{2}^{0,n+1}:\Z_{2}\to \Z_{2}^{\sigma_{n+1}}$
$$d_{2}^{0, n+1}(x)=x(\overline{m}_{1}, \ldots, \overline{m}_{\sigma_{\mu}})$$
where $m_{k}=m_{[\alpha_{k}]}.$

\end{teo}
\begin{proof}
We start with
$$d_{2}:E_{2}^{2,\nu}\simeq \tilde{H}^{1}(h^{-1}(\Sigma_{\nu+1}))\to E_{2}^{4, \nu-1}=\Z_{2}$$
which is given by $x\mapsto (x\smile\gamma_{1,\nu})|_{(B, \Omega^{\nu})}.$ In order to do that we choose a small neighborhood $U(\eps)$ of $\Sigma_{\nu+1}=\{[\beta_{1}], \ldots, [\beta_{\sigma_{\nu+1}}]\}$ and we define $C(\eps)=h^{-1}(U(\eps)).$ If we set $\gamma_{1, \nu}(\eps)=\gamma_{1,\nu}|_{(B, C(\eps))},$ then since $C(\eps)\cup \Omega^{\nu+1}=\Omega^{\nu}=S^{3},$
$$d_{2}^{2, \nu}(x)=(x\smile \gamma_{1,\nu}(\eps))|_{(B, S^{3})}.$$
We let $\partial^{*}c_{1},\ldots, \partial^{*}c_{\sigma_{\nu+1}}$ be the generators of $H^{2}(B,C(\eps))\stackrel{\partial^{*}}{\simeq}H^{1}(C(\eps)),$ where $c_{k}$ is the dual of $h^{-1}[\beta_{k}],\, k=1, \ldots, \sigma_{\nu+1}.$ Lemma \ref{split} implies now that $w_{1}(L_{[\beta_{i}]})=\nu\, \textrm{mod}\, 2$ because $\nu=\min \ii^{+}=\textrm{rk}(p(\beta_{k}))$ for every $k=1, \ldots, \sigma_{\nu+1}.$ It follows that
$$\gamma_{1, \nu}(\eps)=\overline{\nu}\sum_{k=1}^{\sigma_{\nu+1}}\partial^{*}c_{k}.$$
If we let now $\partial^{*}g_{1}, \ldots, \partial^{*}g_{\sigma_{\nu+1}}$ be the generators of $H^{2}(B, \Omega^{\nu+1})\stackrel{\partial^{*}}{\simeq}H^{1}(\Omega^{\nu+1}),$ where $g_{k}=\textrm{lk}(\,\cdot\,, h^{-1}[\beta_{k}]),\, k=1, \ldots, \sigma_{\nu+1},$ we have the following formula
$$d_{2}^{2, \nu}(x)=\overline{\nu}\sum_{k={1}}^{\sigma_{\nu+1}}x^{k}\quad\quad x=\sum_{k=1}^{\sigma_{\nu+1}}x^{k}\partial^{*}g_{k}.$$
We assume now that $\mu=n+1$ and we compute $$d_{2}:E_{2}^{0,n+1}\simeq \Z_{2}\to E_{2}^{2, n}\simeq \tilde{H}_{1}(h^{-1}(\Sigma_{n+1})).$$
Consider thus $\Sigma_{n+1}=\{[\alpha_{1}],\ldots, [\alpha_{\sigma_{n+1}}]\}$ and let $f_{1},\ldots, f_{\sigma_{n+1}}$ be the generators of $\tilde{H}^{1}(S^{3}\backslash h^{-1}(\Sigma_{n+1})):$
$$f_{k}(c)=\textrm{lk}(c, h^{-1}[\alpha_{k}])\quad \forall c\in \tilde{H}_{1}(S^{3}\backslash h^{-1}(\Sigma_{n+1})).$$
In this way we have
$$H^{2}(B, \Omega^{n+1})=\langle \partial^{*}f_{1},\ldots, \partial^{*}f_{\sigma_{n+1}}\rangle.$$
It is shown in \cite{AgLe} that $$w_{1,n+1}=p^{*}\textrm{lk}(\, \cdot\,, \{\lambda_{n+1}=\lambda_{n}\})$$
In our case $p^{-1}\{\lambda_{n+1}=\lambda_{n+2}\}=h^{-1}(\Sigma_{n+1}):$ if $\alpha \notin h^{-1}(\Sigma_{n+1}),$ then $\textrm{rk}(p(\alpha))=n+1$ and thus $\ii^{+}(p(\alpha))=n+1$ and $\lambda_{n+1}(p(\alpha))>\lambda_{n+2}(p(\alpha));$ on the contrary if $\alpha\in h^{-1}(\Sigma_{n+1}),$ then $\textrm{rk}(p(\alpha))\leq n$ and $\lambda_{n+1}(p(\alpha))=\lambda_{n+2}(p(\alpha))=0.$ Since $\gamma_{1, n+1}=\partial^{*}w_{1,n+1},$ then we have
$$d_{2}^{0,n+1}(1)=\gamma_{1, n+1}=\sum_{k=1}^{\sigma_{n+1}}\overline{m}_{k}\partial^{*} f_{k}$$
where $m_{k}=m_{[\alpha_{k}]}$ comes from the fact that we are taking the pull-back of the class $\textrm{lk}(\, \cdot\,, \{\lambda_{n+1}=\lambda_{n+2}\})$ through $p$ and multiplicities have to be taken into account.\\

\end{proof}

\begin{remark}
Notice that if $\mu=n+1$ and $\nu=n, $ then $$d_{2}^{2,n}\circ d_{2}^{0, n+1}(1)=\overline{n}\sum_{i}\overline{m_{i}}=\overline{n(n+1)}=0.$$ 
\end{remark}

\begin{remark}
Consider the bundle $\R^{\mu}\to L_{\mu}\to D_{\mu}$ as defined in the second section and its projectivization $\RP^{\mu-1}\to P_{\mu}\stackrel{q}{\rightarrow} D_{\mu}.$ Since $L_{\mu}\subset D_{\mu}\times \R^{2n+2},$ then $P_{\mu}\subset D_{\mu}\times \RP^{2n+1}$ and the restriction of the projection on the second factor 
$$l:P_{\mu}\to \RP^{2n+1}$$ is a map which is a linear embedding on the fibres. It is not difficult to prove that for this map we have $\textrm{rk}(l^{*})_{k}\leq 1-\textrm{rk}({i_{R}}{*})_{2n+1-k}$ (see \cite{AgLe}). Thus by Theorem C of \cite{AgLe} we have the following implication:
$$\textrm{rk}(l^{*})_{k}=1 \Rightarrow E_{\infty}^{0,k}=0.$$
Using the fact that $l$ is linear on the fibres, we can compute $l^{*}x^{\mu}$ where $x$ is the generator of $H^{1}(\RP^{2n+1};\Z_{2})$ and $l^{*}x=y:$
$$y^{\mu}=(w_{1}(L_{\mu})y+w_{2}(L_{\mu}))y^{\mu-2}$$ where we interpret $w_{i}(L_{\mu})$ as a class on $P_{\mu}$ via $q^{*}.$
Thus we see that
$$(w_{1}(L_{\mu})\neq0\,\, \,\textrm{or}\,\,\, w_{2}(L_{\mu})\neq 0)\Rightarrow E_{\infty}^{0,\mu}=0.$$
Applying the same reasoning and computing $y^{k}$ for $k\geq \mu+1$ we get similar conditions for the vanishing of $E_{\infty}^{0,k}.$ Notice in particular that $w_{1}(L_{\mu})=w_{1,\mu},$ hence if it is nonzero $d_{2}^{0,\mu}$ also is nonzero and $E_{3}^{0,\mu}=0,$ which consequently gives $E_{\infty}^{0,\mu}=0.$ Such considerations suggest that higher differential $d_{r}^{0,*}$ for $(E_{r},d_{r})$ are closely related to higher characteristic classes.
\end{remark}

We get as a corollary of the previous theorem the following well known fact from plane geometry.

\begin{coro}The intersection of two quadrics in $\C\emph{P}^{2}$ consists of four points if and only if the associated pencil has exactly three singular elements.
\end{coro}
\begin{proof}
Notice that a for a pencil of quadrics in $\CP^{2}$ generated by $Q_{0},Q_{1}$ the following four possibilities can happen for $$\{[\alpha_{0},\alpha_{1}]\in \CP^{1}\, |\, \det(\alpha_{0}Q_{0}+\alpha_{1}Q_{1})=0\}=\left\{ \begin{array}{ccc}\CP^{1}&(\infty)\\
\textrm{one point}& (1)\\
\textrm{two points}& (2)\\
\textrm{three points}& (3)\\

\end{array}\right.$$
The general table for the ranks of $E_{2}(R)$ has the following picture:

$$
\begin{array}{|c|c|c|c|c|}
\hline
1&0&0&0&0\\
1&0&0&0&0\\
1&0&0&0&0\\
a&0&c&c'&0\\
b&0&d&d'&f\\
0&0&e&e'&g\\

\hline

\end{array}
$$
Now $b_{0}(C)=b_{0}(R)\leq 1+c'+f$ and \begin{itemize}
\item [$(\infty)$]: $a=1,\,c=c'=f=0$ and $b_{0}(C)=1.$
\item [$(1),(2)$]: $a=b=0,\,c'=c-1\leq1,\, f\leq 1$ and $b_{0}(C)\leq 3.$
\item [$(3)$]: $a=b=0,\, c=3,\, c'=2,\, f=1$ and by Theorem \ref{d2} $d_{2}^{2,2}$ is identically zero ($\nu=2$ is even); also, since $E_{\infty}^{0,5}=\Z_{2}$, Theorem \ref{coo} implies $E_{\infty}^{0,4}=\Z_{2}$ (the number of $\Z_{2}$ summands in $E_{\infty}^{0,*}$ is even); thus $d_{2}^{0,4}=d_{3}^{0,4}=d_{4}^{0,4}\equiv 0$ and $b_{0}(C)=4.$

\end{itemize}
It follows that
$$b_{0}(C)=4 \iff (3).$$
\end{proof}

\begin{example}[The complete intersection of two quadrics]
We recall from \cite{Reid} that the condition for $C=V(q_{0},q_{1})$ to be a complete intersection is equivalent to have $\mu=n+1,\, \sigma_{\mu}=n+1$ and $\nu=n.$ In other words the equation $\det (\alpha_{0}Q_{0}+\alpha_{1}Q_{1})=0$ must have $n+1$ distinct roots and at each root $[\alpha_{0},\alpha_{1}]$ the pencil must by simply degenerate, i.e. the rank of $\alpha_{0}Q_{0}+\alpha_{1}Q_{1}$ must be $n$ (notice in particular that for the case $n=2$ we have the above result).\\
Thus the table for the rank of $E_{2}$ is the following:

$$
\textrm{rk}(E_{2})=\begin{array}{c|c|c|c|c|c|}
2n+1&1&0&&&\\
&\vdots&\vdots&&&\\
&1&0&&&\\
n+1&1&0&&&\\
\hline
n&0&0&n+1&n&0\\
&0&0&0&0&1\\
&\vdots&\vdots&\vdots&\vdots&\vdots\\
&0&0&0&0&1\\
\hline

\end{array}
$$
We distinguish the two cases $n$ even and $n$ odd.
\begin{itemize}
\item [($n$ even)]: In this case, by Theorem \ref{d2}, $d_{2}^{0,n+1}$ is injective and $d_{2}^{2,n}$ is zero. Hence the table for the rank of $E_{3}$ is the following:
$$
\textrm{rk}(E_{3})=\begin{array}{c|c|c|c|c|c|}
2n+1&1&0&&&\\
&\vdots&\vdots&&&\\
&1&0&&&\\

&1&0&0&0&0\\
n+1&0&0&0&0&0\\
\hline
n&0&0&n&n&0\\
&0&0&0&0&1\\
&\vdots&\vdots&\vdots&\vdots&\vdots\\
&0&0&0&0&1\\
\hline

\end{array}
$$
Since $d_{3}^{0,n+3}=d_{4}^{0,n+3}=0$ then $E_{\infty}^{0,n+3}=\Z_{2};$ since $n$ is even, then by Theorem \ref{coo} we have $E_{\infty}^{0,n+2}=E_{\infty}^{0,n+3}=\Z_{2}$ and thus $d_{3}^{0,n+2}=d_{4}^{0,n+2}=0.$ This implies $$E_{3}=E_{\infty}.$$
Thus the $\Z_{2}$-Betti numbers of $R$ are:
$$
b_{j}(R)=\left\{ \begin{array}{cc} 1&\textrm{if $j\neq n-2,n-1,$ $0\leq j\leq 2n-1$;}\\
n+2 &\textrm{if $j=n-2,n-1$}
\end{array}\right.
$$
Consequently the $\Z_{2}$-Betti numbers of $C$ are:
$$
b_{j}(C)=\left\{ \begin{array}{cc} 0&\textrm{if $j$ is odd;}\\
1 &\textrm{if $j$ is even, $j\neq n-2$ and $0\leq j\leq 2n-2$}\\
n+2& \textrm{if $j=n-2$}
\end{array}\right.
$$
\item [($n$ odd)]: in this case, by Theorem \ref{d2}, $d_{2}^{0,n+1}$ is injective and $d_{2}^{2,n}$ is surjective. Thus the table for the rank of $E_{3}$ is the following:
$$\textrm{rk}(E_{3})=
\begin{array}{c|c|c|c|c|c|}
2n+1&1&0&&&\\
&\vdots&\vdots&&&\\
&1&0&&&\\

&1&0&0&0&0\\
n+1&0&0&0&0&0\\
\hline
n&0&0&n-1&n&0\\
&0&0&0&0&0\\
&0&0&0&0&1\\

&\vdots&\vdots&\vdots&\vdots&\vdots\\
&0&0&0&0&1\\
\hline

\end{array}
$$
Since $E_{\infty}^{0,n+1}=0$ and $n$ is odd, then by Theorem \ref{coo} we have $E_{\infty}^{0,n+2}=0$,  thus $d_{3}^{0,n+2}$ must be injective and the table of rank of $E_{4}=E_{\infty}$ must be the following:
 
$$\textrm{rk}(E_{4})=\textrm{rk}(E_{\infty})=
\begin{array}{c|c|c|c|c|c|}
2n+1&1&0&&&\\
&\vdots&\vdots&&&\\
&1&0&&&\\

&0&0&0&0&0\\
n+1&0&0&0&0&0\\
\hline
n&0&0&n-1&n-1&0\\
&0&0&0&0&0\\
&0&0&0&0&1\\

&\vdots&\vdots&\vdots&\vdots&\vdots\\
&0&0&0&0&1\\
\hline

\end{array}
$$
Thus the $\Z_{2}$-Betti numbers of $R$ are:
$$
b_{j}(R)=\left\{ \begin{array}{cc} 1&\textrm{if $j\neq n-2,n-1,$ $0\leq j\leq 2n-1$;}\\
n&\textrm{if $j=n-2,n-1$}
\end{array}\right.
$$
Consequently the $\Z_{2}$-Betti numbers of $C$ are:
$$
b_{j}(C)=\left\{ \begin{array}{cc} 0&\textrm{if $j$ is odd and $j\neq n-2$;}\\
1 &\textrm{if $j$ is even, $0\leq j\leq 2n-2$}\\
n-1 & \textrm{if $j=n-2$}
\end{array}\right.
$$

\end{itemize}
Thus the complete intersection of two quadrics $C$ in $\CP^{n}$ has complex dimension $m=n-2$ and its $m$-th Betti number is $m+4$ if $m$ is even and $m+1$ if $m$ is odd.
\end{example}
\begin{example}
Consider the two quadrics $$q_{0}(z_{0},z_{1},z_{2},z_{3})=z_{0}z_{2}-z_{1}^{2}\quad\textrm{and}\quad q_{1}(z_{0},z_{1},z_{2},z_{3})=z_{0}z_{3}-z_{1}z_{2}.$$
Then $\det(\alpha_{0}{Q_{0}}+\alpha_{1}Q_{1})=\alpha_{1}^{4}$ and $\textrm{rk}(\alpha_{0}Q_{0}+\alpha_{1}Q_{1})\equiv4$ except at the point $[1,0]$ where we have $\textrm{rk}(Q_{0})=3.$ Notice in this case that $\textrm{rk}(p([1,0]))\neq n+1-m_{[\alpha]}=4-m_{[\alpha]}=0.$
The table for the rank of $E_{2}$ has the following picture:
$$
\textrm{rk}(E_{2})=\begin{array}{|c|c|c|c|c|}
\hline
1&0&0&0&0\\
1&0&0&0&0\\
1&0&0&0&0\\
1&0&0&0&0\\
0&0&1&0&0\\
0&0&0&0&1\\
0&0&0&0&1\\
0&0&0&0&1\\

\hline

\end{array}
$$

Since $\mu=4=n+1,$ then we can use the previous formula for $d_{2}^{0,\mu}$ and since we have $m_{[1,0]}=4$ it follows $d_{2}^{0,4}\equiv 0.$ On the other hand $d_{2}^{2,4}$ is multiplication by $\nu=3 \, \textrm{mod}\, 2$, hence it is an isomorphism. Hence the table for the rank of $E_{3}$ has the following picture:

$$
\textrm{rk}(E_{3})=\begin{array}{|c|c|c|c|c|}
\hline
1&0&0&0&0\\
1&0&0&0&0\\
1&0&0&0&0\\
1&0&0&0&0\\
0&0&0&0&0\\
0&0&0&0&0\\
0&0&0&0&1\\
0&0&0&0&1\\

\hline

\end{array}
$$

Since $d_{3}^{0,5}=d_{4}^{0,5}\equiv 0,$ then $E_{\infty}^{0,5}=\Z_{2}$ and Theorem \ref{coo} implies also $E_{\infty}^{0,4}=\Z_{2}.$
Thus $E_{3}=E_{4}=E_{\infty}.$ Hence, for the only possible nonzero Betti numbers of $R$ we have 
$b_{0}(R)=b_{1}(R)=1,\, b_{2}(R)=b_{3}(R)=2.$ This implies the following for the Betti numbers of $C:$
$$b_{0}(C)=1,\, b_{2}(C)=2\quad \textrm{and}\quad  b_{i}(C)=0,\, i\neq 0,2.$$
Using Theorem \ref{coo} we see that $(i_{C}^{*})_{0}$ and $(i_{C})^{*}_{2}$ are injective.\\
Looking directly at the equations for $C$ we see that it equals the union of the skew-cubic and a (complex projective) line meeting at one point; thus topologically $C\sim S^{2}\vee S^{2}.$
\end{example}
\begin{example}
Consider the two quadrics $$q_{0}(z_{0},z_{1},z_{2})=z_{0}^{2}-z_{1}^{2}\quad\textrm{and} \quad q_{1}(z_{0},z_{1},z_{2})=2z_{0}(z_{1}+z_{2}).$$
We have $\det (\alpha_{0}Q_{0}+\alpha_{1}Q_{1})\equiv 0$ and $\textrm{rk}(\alpha_{0}Q_{0}+\alpha_{1}Q_{1})=2$ for every $[\alpha_{0},\alpha_{1}]\in \CP^{1}.$
Thus the table for the rank of $E_{2}$ is:

$$
\textrm{rk}(E_{2})=\begin{array}{|c|c|c|c|c|}
\hline
1&0&0&0&0\\
1&0&0&0&0\\
1&0&0&0&0\\
1&0&0&0&0\\
0&0&0&0&1\\
0&0&0&0&1\\

\hline

\end{array}
$$
By dimensional reasons, the only possibly nonzero differential is $d_{4}.$ Since $d_{4}^{0,2}=0,$ then $E_{\infty}^{0,2}=\Z_{2}$ and by Theorem \ref{coo} also $E_{\infty}^{0,3}=\Z_{2}.$
Since $E_{\infty}^{0,3}=\Z_{2},$ then $d_{4}^{0,3}=0;$ hence $E_{\infty}^{4,0}=\Z_{2}.$ On the other side we have $d_{4}^{0,5}=0$ and hence $E_{\infty}^{0,5}=\Z_{2};$ Theorem \ref{coo} implies $E_{\infty}^{0,4}=\Z_{2}.$ Since $E_{\infty}^{0,4}=\Z_{2},$ then $d_{4}^{0,4}=0$ and $E_{\infty}^{4,1}=\Z_{2}.$\\
All this tells us that $E_{\infty}=E_{2}.$ The only possible nonzero Betti numbers of $R$ are
$b_{0}(R)=b_{1}(R)=2,\, b_{2}(R)=b_{3}(R)=1.$ This implies the following for the Betti numbers of $C:$
$$b_{0}(C)=2,\, b_{2}(C)=1\quad \textrm{and}\quad  b_{i}(C)=0,\, i\neq 0,2.$$
Looking directly at the equations of $C$ we see that it equals the union of the point $[1,1,0]$ and the complex projective line $\{z_{0}+z_{1}=0\}.$ 
\end{example}

\end{document}